\documentclass[12pt]{article}
\usepackage{amssymb}
\usepackage{amsmath}
\usepackage{amsfonts}

\newtheorem{theorem}{Theorem}

\newtheorem{definition}[theorem]{Definition}
\newtheorem{example}[theorem]{Example}

\newtheorem{lemma}[theorem]{Lemma}

\newtheorem{proposition}[theorem]{Proposition}

\newenvironment{proof}[1][Proof]{\noindent\textbf{#1.} }{\ \rule{0.5em}{0.5em}}
\begin{document}

\title{Minimal driver sets on path and cycle graphs\\
with arbitrary non-zero weights}
\author{Johannes G. Maks \\
Delft Institute of Applied Mathematics \\
Delft Unviversity of Technology\\
Mekelweg 4, 2628 CD Delft, The Netherlands\\
Email: j.g.maks@tudelft.nl}
\date{}
\maketitle

\begin{abstract}
Let $G$ be a simple, undirected graph on the vertex set $V=\{1,2,\ldots ,n\}$
and let $A$ be the adjacency matrix of $G.$ A non-empty subset $%
\{i_{1},i_{2},\ldots ,i_{k}\}$ of $V$ is called a driver set for $G$ if the
system $\mathbf{\dot{x}}=A\mathbf{x}+u_{1}\mathbf{e}_{i_{1}}+\cdots +u_{k}%
\mathbf{e}_{i_{k}}$ is controllable.

In this paper we classify the minimal driver sets for the path and cycle
graphs $P_{n}$ and $C_{n}$ for all values of $n$ and we determine which of
those minimal driver sets render the system to be strongly structural
controllable with respect to the family of all symmetric matrices $X$
satisfying $x_{ij}=0\Leftrightarrow a_{ij}=0.$

Note that this new type of strong structural controllability requires all
diagonal elements of the system matrix to be equal to zero so for example
the Laplacian matrix is not included in the family. $\bigskip $

\bigskip

\end{abstract}

\section{Introduction}

Let $G=(V,E)$ be a simple, undirected graph on the vertex set $%
V=\{1,2,\ldots ,n\}$ with adjacency matrix $A.$ For each non-empty subset $%
S=\{i_{1},i_{2},\ldots ,i_{k}\}$ of $V$ let $B_{S}$ be the $(n\times k)$%
-matrix with columns $\mathbf{e}_{i_{1}},$ $\mathbf{e}_{i_{2}},\ldots ,%
\mathbf{e}_{i_{k}}$. A non-empty subset $S$ is called a driver set for $G$
if the system $\mathbf{\dot{x}}=A\mathbf{x}+B_{S}\mathbf{u}$, or the pair $%
(A,B_{S}),$ is controllable.

Let Sym$(G)$ be the set of symmetric $(n\times n)$-matrices $X$ with free
diagonal elements and off-diagonal elements $x_{ij}$ unequal to zero if and
only if $(i,j)\in E.$ If $(A,B_{S})$ is controllable then $(X,B_{S})$ is
controllable not just for $X=A$ but for almost all $X\in $ Sym$(G)$, a
property which is referred to as \textit{structural controllability} in the
literature. The subject of structural controllability of networks has been
studied intensively during the last two decades by many researchers in the
systems and control community, in view of applications where the weights of
the edges are not fixed due to lack of information or numerical instability.

A stronger version of structural controllability is the property that $%
(X,B_{S})$ is controllable for all $X\in $ Sym$(G).$ This property is
referred to as \textit{strong structural controllability} in the literature.

Perhaps surprisingly, it turned out that this notion of strong structural
controllability of a network is connected to the notion of a zero forcing
set of the underlying graph.\footnote{%
The notion of a zero forcing set (briefly summarized in section 4) had been
introduced several years earlier in a different context \cite{AIM}.} It has
been proved in \cite{Monshizadeh 2014} that $(X,B_{S})$ is controllable for
all $X\in $ Sym$(G)$ if and only if $S$ is a zero forcing set of $G.$

A zero forcing set is a special type of driver set but in general not every
driver set is a zero forcing set. The discovery of the connection between
strong structural controllability and zero forcing sets has understandably
caused a surge of research in the latter. We believe there are several good
reasons for studying all minimal driver sets for the system $(A,B_{S}),$
such as the following:

\begin{itemize}
\item The minimal size of a driver set could be smaller than the minimal
size of a zero forcing set (see example \ref{Q3} in section 3), which could
be relevant in applications where using a driver set of minimum cardinality
is essential.

\item Additional requirements about the relative positions of the vertices
in a driver set may exist which might not be satisfied by the zero forcing
sets.

\item Strong structural controllability with respect to Sym$(G)$ allows for $%
\left\vert V\right\vert +\left\vert E\right\vert $ degrees of freedom in the
system matrix. For simple graphs $G$ it seems more natural to study strong
structural controllability with respect to the smaller family Sym$_{0}(G)$
consisting of all matrices in Sym$(G)$ with zeros on the diagonal, allowing
for $\left\vert E\right\vert $ degrees of freedom only. Driver sets $S$ for
which $(X,B_{S})$ is controllable for all $X\in $ Sym$_{0}(G)$ are not
necessarily a zero forcing set.
\end{itemize}

In this paper we determine all minimal driver sets for the path and cycle
graphs for all values of $n,$ using a simple controllability test in terms
of the eigenspaces of the adjacency matrices. We also determine for which of
those minimal driver sets the system is strongly structural controllable
with respect to the family Sym$_{0}(G)$. It will turn out that not all such
sets are zero forcing sets, so we have discovered new types of minimal
driver sets that render the systems to be controllable for all non-zero
weights on the edges of the path and cycle graphs. These (non-trivial)
results for the path and cycle graphs could provide ideas for a similar
classification of minimal driver sets for other types of simple graphs.

The organization of the paper after the introduction is as follows. In
section 2 we present some relevant background information and notations. In
section 3 we derive a necessary and sufficient condition for controllability
of a system on a graph in terms of the eigenspaces of the adjacency matrix
of the graph and give three illustrative examples. In section 4 we introduce
a new type of strong structural controllability which we believe to be
natural for systems on simple graphs. In sections 5 and 6 we present our
results about the minimal driver sets for the path and cycle graphs.

\section{Preliminaries and notations}

\subsection{Controllability of linear systems}

Let $A$ and $B$ be matrices of sizes $(n\times n)$ and $(n\times k),$
respectively. A system $\mathbf{\dot{x}}=A\mathbf{x}+B\mathbf{u,}$ or the
pair $(A,B),\,$is controllable if any initial state vector can be steered by
the system to any other state vector in finite time. There are several
equivalent ways to state the Popov-Belevich-Hautus (PBH) controllability
test. Each of the following four properties is a necessary and sufficient
condition for $(A,B)$ to be controllable:%
\begin{equation*}
\begin{array}{l}
1.\text{ rank }\left[
\begin{array}{cc}
A-\lambda I & B%
\end{array}%
\right] =n\text{ for all }\lambda \in \mathbb{C} \\
2.\text{ rank }\left[
\begin{array}{cc}
A-\lambda I & B%
\end{array}%
\right] =n\text{ for all eigenvalues }\lambda \text{ of }A \\
3.\text{ no eigenvector }\mathbf{v}\text{ of }A^{T}\text{ exists with }B^{T}%
\mathbf{v}=0 \\
4.\text{ Nul }B^{T}\cap E_{\lambda }=\left\{ \mathbf{0}\right\} \text{ for
each eigenspace }E_{\lambda }\text{ of }A^{T}%
\end{array}%
\end{equation*}%
We shall refer to these statements as PBH 1, \ldots , PBH 4 in the sequel of
this paper. The first three condtions are well-known and used very often in
the literature. We add PBH 4 to the list because it will turn out to be
useful in this paper (see section 3).

PBH 2 implies that rank $B\geq $ gm $(\lambda )$ for each eigenvalue $%
\lambda $ of $A,$ hence%
\begin{equation}
\text{rank }B\geq \text{ }\underset{\lambda \in \sigma (A)}{\text{max}}\text{
}\left\{ \text{geometric multiplicity }\lambda \right\} .
\label{lower bound rank B}
\end{equation}

Two systems $(A,B)$ and $(A^{\prime },B^{\prime })$ are called equivalent if
there exists an invertible matrix $T$ such that%
\begin{equation*}
\left\{
\begin{array}{l}
A^{\prime }=TAT^{-1} \\
B^{\prime }=TB%
\end{array}%
\right.
\end{equation*}

If two systems are equivalent then controllability of the one is equivalent
to controllability of the other.

\subsection{Graphs}

An undirected graph $G=(V,E)$ consists of a set $V=\{1,2,\ldots ,n\}$ and a
set $E$ of unordered pairs $\{i,j\}$ of vertices. The elements of $V$ and $E$
are called the vertices and edges of $G,$ respectively. Two vertices $i$ and
$j$ are called adjacent if $\{i,j\}\in E.$ In this paper we only consider
graphs without loops, i.e., graphs without edges of the form $\{i,i\};$ such
graphs are called simple graphs. A path of length $k$ between two vertices $%
i $ and $j$ is a sequence of vertices $i_{1}=i,i_{2},\ldots ,i_{k-1},i_{k}=j$
such that $i_{t}$ and $i_{t+1}$ are adjacent for all $t\in \{1,2,\ldots
,k-1\}.$ The distance between two vertices $i$ and $j$ in a graph, denoted
by $d(i,j),$ is the shortest length of a path between $i$ and $j.$ The
adjacency matrix of a graph $G=(V,E)$ with $V=\{1,2,\ldots ,n\}$ is the
symmetric $(n\times n)$-matrix $A=[a_{ij}]$ with $a_{ij}=1$ if $\{i,j\}\in E$
and $a_{ij}=0$ otherwise.

An automorphism of a graph $G=(V,E)$ is a permutation $\sigma $ of $V$ which
satisfies the property $\{\sigma (i),\sigma (j)\}\in E$ if and only if $%
\{i,j\}\in E$. The set of all automorphisms of $G$ forms a group and is
denoted by $Aut(G).$ Every element of $Aut(G)$ can be represented uniquely
by a permutation matrix $P$ satisfying $A=P^{T}AP.$

In this paper we will pay special attention to two special graphs with
vertex set $V=\{1,2,\ldots ,n\},$ viz. the path graphs denoted by $P_{n\text{
}}$and the cycle graphs denoted by $C_{n}$. The edge sets are $%
\{\{1,2\},\{2,3\},\ldots ,\{n-1,n\}\}$ for $P_{n}$ and $\{\{1,2\},\{2,3\},%
\ldots ,\{n-1,n\},\{1,n\}\}$ for $C_{n}.$ The automorphism group $Aut(P_{n})$
is generated by the reflection $\sigma $ defined by $\sigma
(v_{i})=v_{n+1-i} $ for all $i\in \{1,2,\ldots ,n\},$ hence $Aut(P_{n})\cong
\mathbb{Z}_{2}.$ The automorphism group of $C_{n}$ is generated by the
rotation $\sigma =(1,2,\ldots ,n)$ and the reflection $\tau $ about the axis
that passes through the vertex $1$ and the centre of $C_{n}$, hence $%
Aut(C_{n})\cong D_{2n},$ the dihedral group of order $2n$ (the group of
symmetries of a regular $n$-gon$).$

\subsection{Pl\"{u}cker coordinates of subspaces}

Let $W$ be an $m$-dimensional subspace of $\mathbb{R}^{n}$ and $\left\{
\mathbf{w}_{1},\mathbf{w}_{2},\ldots ,\mathbf{w}_{m}\right\} $ a basis of $%
W. $ The $\binom{n}{m}$ maximal minors of the $(n\times m)$-matrix $\left[
\begin{array}{cccc}
\mathbf{w}_{1} & \mathbf{w}_{2} & \cdots & \mathbf{w}_{m}%
\end{array}%
\right] $ are the Pl\"{u}cker coordinates of the subspace $W.$ These
coordinates are homogeneous coordinates as they are determined up to a joint
non-zero factor: if we change the basis of $W$ all Pl\"{u}cker coordinates
are multiplied by the determinant of the $(m\times m)$-matrix that
represents the change of basis. The Pl\"{u}cker coordinates of $W$ are
indexed by the $\binom{n}{m}$ sets $\{i_{1},i_{2},\ldots ,i_{i_{m}}\}$ with $%
1\leq i_{1}\leq i_{2}\leq \cdots \leq i_{m}\leq n.$

\section{Minimal driver sets on graphs}

Let $G=(V,E)$ be a simple, undirected graph on the vertex set $%
V=\{1,2,\ldots ,n\}$ with adjacency matrix $A=A(G).$ For each non-empty
subset $S=\{i_{1},i_{2},\ldots ,i_{k}\}$ of $V$ let $B_{S}$ be the $(n\times
k)$-matrix with columns $\mathbf{e}_{i_{1}},$ $\mathbf{e}_{i_{2}},\ldots ,%
\mathbf{e}_{i_{k}}$. We start with some definitions.

\begin{definition}
A non-empty subset $S$ of $V$ is called a driver set for $G$ if the system $%
(A,B_{S})$ is controllable.
\end{definition}

\begin{definition}
$D(G)$ denotes the minimum cardinality of a driver set for the graph $G$.
\end{definition}

\begin{definition}
$N_{D}(G)$ denotes the number of minimal driver sets for $G.$
\end{definition}

\begin{definition}
$M(G)$ denotes the maximum of all geometric multiplicities of the adjacency
matrix $A(G).$
\end{definition}

Since rank $B_{S}=\left\vert S\right\vert $ inequality (\ref{lower bound
rank B}) yields%
\begin{equation}
D(G)\geq M(G).  \label{lower bound D(G)}
\end{equation}%
Application of PBH 4 to the pair $(A,B_{S})$ with $A=A(G)$ (which is a
symmetric matrix hence $A^{T}$ can be replaced by $A$) yields the statement
\begin{equation*}
S\text{ is a driver set for }G\Longleftrightarrow \text{Nul }B_{S}^{T}\cap
E_{\lambda }=\left\{ \mathbf{0}\right\}
\end{equation*}%
for each of the eigenspaces $E_{\lambda }$ of $A.$

Let $W$ be an $m$-dimensional subspace of $\mathbb{R}^{n}$ and $S$ a subset
of $V$ with $\left\vert S\right\vert =k\geq 1.$ Then Nul $B_{S}^{T}$ $\cap $
$W=\left\{ \mathbf{0}\right\} $ if and only if $\left\{ B_{S}^{T}\mathbf{w}%
_{1},B_{S}^{T}\mathbf{w}_{2},\ldots ,B_{S}^{T}\mathbf{w}_{m}\right\} $ is
linearly independent for each basis $\left\{ \mathbf{w}_{1},\mathbf{w}%
_{2},\ldots ,\mathbf{w}_{m}\right\} $ of $W$. This condition can be
rephrased as
\begin{equation*}
\text{rank }B_{S}^{T}M=m
\end{equation*}%
for each $(n\times m)$-matrix $M$ which satisfies Col $M=W$ (i.e., the
columns of $M$ form a basis of $W$). Note that $B_{S}^{T}M$ is a $(k\times
m) $-matrix and that the condition rank $B_{S}^{T}M=m$ implies $k\geq m,$
i.e.,%
\begin{equation*}
\left\vert S\right\vert \geq \dim W.
\end{equation*}

If $\left\vert S\right\vert =\dim W=k$ then $B_{S}^{T}M$ is a $(k\times k)$%
-matrix in which case the condition rank $B_{S}^{T}M=k$ is equivalent to the
condition det $B_{S}^{T}M\neq 0.$

The determinant of $B_{S}^{T}M$ is the homogeneous Pl\"{u}cker coordinate
indexed by $S$ of the subspace $W.$ Hence we have the following lemma, which
is a useful tool for constructing minimal driver sets for graphs $G$ with $%
D(G)=M(G).$

\begin{lemma}
\label{Plucker}Let $G$ be graph with $D(G)=M(G)=k.$ If $S$ is a minimal
driver set for $G$ then for each $k$-dimensional eigenspace $E_{\lambda }$
of $A(G)$ the Pl\"{u}cker coordinate of $E_{\lambda }$ indexed by $S$ is
unequal to zero.
\end{lemma}

\begin{example}
$G=P_{5}.$ The eigenvalues of $A$ are $-1,0,1,-\sqrt{3},\sqrt{3}.$ Basis
vectors for the corresponding eigenspaces are the columns of the matrix $M$
given by
\begin{equation*}
M=\left[
\begin{array}{rrrrr}
-1 & 1 & -1 & 1 & 1 \\
1 & 0 & -1 & -\sqrt{3} & \sqrt{3} \\
0 & -1 & 0 & 2 & 2 \\
-1 & 0 & 1 & -\sqrt{3} & \sqrt{3} \\
1 & 1 & 1 & 1 & 1%
\end{array}%
\right] .
\end{equation*}
\end{example}

$D(P_{5})=M(P_{5})=1$ and $N_{D}(P_{5})=2.$ The two minimal driver sets are $%
\{1\}$ and $\{5\}$ because rows 1 and 5 of the matrix $M$ do not contain a
zero. The two minimal driver sets lie in a single orbit under the action of
the automorphism group $Aut(P_{5})=\left\langle (1,5)(2,4)\right\rangle .$
The path graphs $P_{n}$ for general $n$ will be discussed in section 5.

\begin{example}
$G=C_{6}.$ The eigenvalues of $A$ are $-2,2,-1^{2},1^{2}$ hence $M(C_{6})=2,$
which implies $D(C_{6})\geq 2.$ Basis vectors for the corresponding
eigenspaces are collected in the following block matrix%
\begin{equation*}
M=\left[ M_{1}|M_{2}|M_{3}|M_{4}\right] =\left[
\begin{tabular}{r|r|rr|rr}
$-1$ & $1$ & $-1$ & $-1$ & $1$ & $-1$ \\
$1$ & $1$ & $0$ & $1$ & $0$ & $-1$ \\
$-1$ & $1$ & $1$ & $0$ & $-1$ & $0$ \\
$1$ & $1$ & $-1$ & $-1$ & $-1$ & $1$ \\
$-1$ & $1$ & $0$ & $1$ & $0$ & $1$ \\
$1$ & $1$ & $1$ & $0$ & $1$ & $0$%
\end{tabular}%
\right] .
\end{equation*}%
By looking at these bases of the eigenspaces we can immediately observe that
$D(C_{6})=2$. It turns out that the nonzero Pl\"{u}cker coordinates of \ $%
E_{-1}$ and $E_{1}$ are precisely the ones indexed by the 12 elements $%
\{i,j\}$ with $d(i,j)\in \{1,2\}.$ Since the basis vectors of the remaining
eigenspaces $E_{-2}$ and $E_{2}$ don't have two zeros in any of these pairs
of positions we can conclude $D(C_{6})=2$ with $N_{D}(C_{6})=12.$ The sets $%
\{i,j\}$ with $d(i,j)=3$ are the sets of cardinality 2 that are not a driver
set. For example $\{1,4\}$ is not a driver set because $\det
B_{\{1,4\}}^{T}M_{3}=0$ or $\det B_{\{1,4\}}^{T}M_{4}=0$ (in this example
both are true):%
\begin{equation*}
B_{\{1,4\}}^{T}M_{3}=\left[
\begin{array}{rr}
-1 & -1 \\
-1 & -1%
\end{array}%
\right] \text{ and }B_{\{1,4\}}^{T}M_{4}=\left[
\begin{array}{rr}
1 & -1 \\
-1 & 1%
\end{array}%
\right] .
\end{equation*}%
The minimal driver sets fall into the two orbits $\{\{i,j\}$ $|$ $d(i,j)=1\}$
and $\{\{i,j\}$ $|$ $d(i,j)=2\}$ under the group $Aut(C_{6})=\left\langle
(1,2,3,4,5,6),(1,2)(3,6)(4,5)\right\rangle .$ The cycle graphs $C_{n}$ for
general $n$ will be discussed in section 6.
\end{example}

Note that in the examples above the property of being a minimal driver set
is invariant under the action of the automorphism group $Aut(G).$ This is
true in general:

\begin{proposition}
Let $\pi \in Aut(G).$ Then $S$ is a driver set for $G$ if and only if $\pi
(S)$ is a driver set for $G.$
\end{proposition}

\begin{proof}
Let $P$ denote the permutation matrix that corresponds to $\pi \in Aut(G).$
Then $B_{\pi (S)}=PB_{S}$ and $A=PAP^{T},$ hence the systems $(A,B_{S})$ and
$(A,B_{\pi (S)})$ are equivalent.
\end{proof}

Suppose we know that $D(G)=k$ and we also know the different orbits of $k$%
-sets under \ the group $Aut(G).$ Then the set all of minimal driver sets
can be simply determined by investigating one representative of each orbit.
The following example illustrates this method.

\begin{example}
\label{Q3}Let $Q_{n},$ $n\geq 1,$ denote the hypercube graph with $2^{n}$
vertices, i.e., the graph with vertex set $V=\{0,1\}^{n}$ and the following
definition of adjacency: $x$ and $y$ are adjacent (form an edge) if and only
if $x$ and $y$ differ in one coordinate position only. The adjacency
matrices of $Q_{n}$ can be defined recursively as follows:%
\begin{equation*}
A(Q_{1})=\left[
\begin{array}{cc}
0 & 1 \\
1 & 0%
\end{array}%
\right] \text{ and }A(Q_{n+1})=\left[
\begin{array}{cc}
A(Q_{n}) & I_{n} \\
I_{n} & A(Q_{n})%
\end{array}%
\right] ,\text{ }n\geq 1.
\end{equation*}%
In this example we consider $Q_{3}.$ The eigenvalues of $A(Q_{3})$ are $%
3^{1},(-3)^{1},1^{3}$ and $(-1)^{3}$ hence $M(G)=3,$ which implies $%
D(Q_{3})\geq 3.$ Basis vectors for the corresponding eigenspaces are
collected in the following block matrix%
\begin{equation*}
M=\left[ M_{1}|M_{2}|M_{3}|M_{4}\right] =\left[
\begin{array}{rrrrrrrr}
1 & -1 & -1 & 0 & 0 & 1 & 0 & 0 \\
1 & 1 & 0 & -1 & 0 & 0 & 1 & 0 \\
1 & 1 & 0 & 0 & -1 & 0 & 0 & 1 \\
1 & -1 & 1 & -1 & -1 & -1 & -1 & -1 \\
1 & 1 & -1 & 1 & 1 & -1 & -1 & -1 \\
1 & -1 & 0 & 0 & 1 & 0 & 0 & 1 \\
1 & -1 & 0 & 1 & 0 & 0 & 1 & 0 \\
1 & 1 & 1 & 0 & 0 & 1 & 0 & 0%
\end{array}%
\right] .
\end{equation*}%
By looking at these bases of the eigenspaces we can immediately observe that
$D(Q_{3})=3:$ the maximal minors of $M_{3}$ and $M_{4}$ from the first three
rows (for example) are both unequal to zero while the first three elements
of $M_{1}$ and of $M_{2}$ are not equal to zero. The minimal size of a zero
forcing set for $Q_{3}$ is equal to $4$ hence none of the minimal driver
sets for $Q_{3}$ is a zero forcing set. Now let's look at the total picture
of minimal driver sets for $Q_{3}.$ There are three orbits of subsets of
vertices of cardinality 3 under the group $Aut(G)\cong S_{3}\times S_{2}^{3}$
(with representatives $\{1,2,3\},\{1,2,4\}$ and $\{1,2,7\}$). It is readily
seen that both $\{1,2,3\}$ and $\{1,2,4\}$ are a minimal driver set and $%
\{1,2,7\}$ is not. The orbits of $\{1,2,3\}$ and $\{1,2,4\}$ have sizes $24$
and $8,$ respectively, hence $N_{D}(Q_{3})=32.$
\end{example}

\section{Strong structural controllability}

Let $S$ be a driver set for a graph $G=(V,E)$ with $\left\vert V\right\vert
=n.$ Let Sym$(G)$ be the set of all symmetric $(n\times n)$-matrices $X=%
\left[ x_{ij}\right] $ satisfying
\begin{equation*}
x_{ij}\neq 0\text{ }\Leftrightarrow (i,j)\in E
\end{equation*}%
for all pairs $(i,j)$ with $i\neq j.$ Hence Sym$(G)$ is the largest set of
symmetric matrices that have their non-zero off-diagonal entries in
precisely the same positions as the adjacency matrix $A.$ The following type
of strong structural controllability is well-known:

\begin{definition}
$(G,S)$ is strongly Sym$(G)$-controllable if $(X,B_{S})$ is \newline
controllable for all $X\in $ Sym$(G).$
\end{definition}

Note that this formulation is a succinct alternative to the more elaborate
version `$(G,S)$ is strongly structurally controllable with respect to Sym$%
(G)$ if $(X,B_{S})$ is controllable for all $X\in $ Sym$(G)$'. More
generally we replace `strongly structurally controllable with respect to $%
F^{\prime }$ by `strongly $F$-controllable' (where $F$ is a set of matrices
having the same zero/non-zero pattern in the off-diagonal entries as $A$).

It has been proved in \cite{Monshizadeh 2014} that $(G,S)$ is strongly Sym$%
(G)$-controllable if and only if $S$ is a zero forcing set of $G.$

The process of zero forcing, which was introduced in \cite{AIM} and
independently in \cite{Burgarth 2007}, can be briefly summarized in the
following way.

Let $S$ be a non-empty subset of vertices of $G$ and suppose all vertices
from $S$ are colored black and all vertices from $V\backslash S$ are colored
white. If there exists a black vertex with exactly one white neighbour $j$
then change the color of $j$ to black and extend the set $S$ to $S\cup \{j\}$
and repeat this process until no color change is possible anymore.

\begin{definition}
The set $S$ is called a zero forcing set if the coloring process described
above results in all vertices being colored black.
\end{definition}

\begin{definition}
The zero forcing number of $G,$ denoted by $Z(G),$ is the \newline
minimum cardinality of a zero forcing set.
\end{definition}

The zero forcing number and minimal zero forcing sets for the path and cycle
graphs are well-known:%
\begin{equation*}
\begin{tabular}{|c|c|c|}
\hline
$G$ & $Z(G)$ & Zero forcing sets $S$ with $\left\vert S\right\vert =Z(G)$ \\
\hline
$P_{n}$ & $1$ & $\{1\}$ and $\{n\}$ \\ \hline
$C_{n}$ & $2$ & $\{i,j\}$ with $d(i,j)=1$ \\ \hline
\end{tabular}%
\end{equation*}

Each zero forcing set is a driver set hence for each graph $G$ we have%
\begin{equation}
D(G)\leq Z(G).  \label{upper bound D(G)}
\end{equation}

Note that for the path and cycle graphs all minimal zero forcing sets lie in
the same orbit under the action of the automorphism groups of the graphs. In
general the minimal zero forcing sets of $G$ could lie in different orbits
but the property of being a zero forcing set is indeed invariant under the
action of $Aut(G).$ This follows immediatly from the definition of a zero
forcing set, which is based on the adjacency structure of $G$ only.
Equivalently we have the following property:

\begin{proposition}
\label{Aut(G)-invariance}Let $\pi \in Aut(G).$ Then $(G,S)$ is strongly Sym$%
(G)$-controllable if and only if $(G,\pi (S))$ is strongly Sym$(G)$%
-controllable.
\end{proposition}

\begin{proof}
Let $P$ denote the permutation matrix that corresponds to $\pi \in Aut(G).$
Then $B_{\pi (S)}=PB_{S}.$ The systems $(X,B_{S})$ and $(PXP^{T},PB_{S})$
are equivalent hence controllability of the one is equivalent to
controllability of the other. On the other hand, Sym$(G)$ is invariant under
the transformation $X\mapsto PXP^{T},$ which permutes the free parameters on
the diagonal and the free parameters on the off-diagonal positions $(i,j)\in
E.$
\end{proof}

Strong Sym$(G)$-controllability allows for $\left\vert V\right\vert
+\left\vert E\right\vert $ degrees of freedom in the system matrix. In
applications with simple graphs is seems more natural to require strong
structural controllability with respect to the smaller family Sym$_{0}(G)$
consisting of all matrices in Sym$(G)$ with zeros on the diagonal, allowing
for $\left\vert E\right\vert $ degrees of freedom only. Driver sets $S$ for
which $(G,S)$ is Sym$_{0}(G)$-controllable are not necessarily a zero
forcing set. Note that Proposition \ref{Aut(G)-invariance} holds for the
smaller family Sym$_{0}(G)$ as well, because the transformation $X\mapsto
PXP^{T}$ doesn't change the zeros on the diagonal.

The chain Sym$_{0}(G)\subset $ Sym$(G)$ gives rise to the following two
types of driver sets $S:$

\begin{definition}
A driver set $S$ is%
\begin{equation*}
\begin{tabular}{l}
of type I if $(G,S)$ is strongly Sym$(G)$-controllable \\
\\
of type II if $\left\{
\begin{tabular}{l}
$(G,S)$ is strongly Sym$_{0}(G)$-controllable, \\
but not strongly Sym$(G)$-controllable%
\end{tabular}%
\right. $%
\end{tabular}%
\end{equation*}
\end{definition}

Driver sets of type I are zero forcing sets, driver sets of type II are not
zero forcing sets but could still be useful for certain applications. Since
each of the two types defined above is $Aut(G)$-invariant we could also
speak of \textit{orbits} of type I, II$.$

To prove that $(G,S)$ is strongly $F$-controllable we can proceed as
follows. Due to PBH 1 $(X,B_{S})$ is controllable for each $X\in $ $F$ if
and only if%
\begin{equation*}
\text{rank }\left[
\begin{array}{cc}
X-\lambda I & B_{S}%
\end{array}%
\right] =n
\end{equation*}%
for all $\lambda \in \mathbb{C}$ and $X\in $ $F.$ The rows of $\left[
\begin{array}{cc}
X-\lambda I & B_{S}%
\end{array}%
\right] $ are linearly independent if and only if the rows of $(X-\lambda
I)_{V\backslash S}$ are linearly independent, where $(X-\lambda
I)_{V\backslash S}$ is the submatrix of $X-\lambda I$ which is obtained by
deleting all rows $i$ with $i\in S.$ Hence $(X,B_{S})$ is controllable for
each $X\in $ $F$ if and only if%
\begin{equation*}
\text{rank }(X-\lambda I)_{V\backslash S}=n-\left\vert S\right\vert
\end{equation*}%
for all $\lambda \in \mathbb{C}$ and $X\in F.$ We shall use this method in
the next two sections where we determine all the orbits of type II minimal
driver sets for the path and cycle graphs.

\section{Path graphs}

Since $Z(P_{n})=1$ and $D(P_{n})\leq Z(P_{n})$ it follows that $D(P_{n})=1$
as well. In the following theorem $\phi $ denotes the Euler totient function.

\begin{theorem}
\label{driver set Pn}$\{i\}$ is a driver set for the graph $P_{n}$ if and
only if%
\begin{equation*}
\text{gcd}(i,n+1)=1,
\end{equation*}%
hence $N_{D}(P_{n})=\phi (n+1).$
\end{theorem}

\begin{proof}
The eigenvalues of $A=A(P_{n})$ are given by $\lambda _{k}=2\cos \left(
\frac{k\pi }{n+1}\right) $ with $k=1,2,\ldots ,n$ and all eigenvalues have
multiplicity equal to $1.$ The vector $\left[
\begin{array}{cccc}
\sin \left( \frac{k\pi }{n+1}\right)  & \sin \left( \frac{2k\pi }{n+1}%
\right)  & \cdots  & \sin \left( \frac{nk\pi }{n+1}\right)
\end{array}%
\right] ^{T}$ is an eigenvector of $A$ belonging to the eigenvalue $\lambda
_{k}.$ Due to PBH 3 (with $A^{T}=A)$ $\{i\}$ is not a driver set if and only
if there exists an eigenvector of $A$ whose $i$-th entry is equal to $0$
hence if and only if $\sin \left( \frac{ik\pi }{n+1}\right) =0$ for at least
one $k\in \left\{ 1,2,\ldots ,n\right\} .$ The latter is true if and only if
$ik\equiv 0$ mod $n+1$ for at least one $k\in \left\{ 1,2,\ldots ,n\right\} ,
$ which is equivalent to gcd $(i,n+1)\neq 1.$
\end{proof}

The orbits of minimal driver sets under the group $Aut(P_{n})\cong S_{2}$
are simply the pairs $\{\{i\},\{n+1-i\}\}$ with gcd $(i,n+1)=1$ hence the
number of orbits is equal to $\frac{1}{2}\phi (n+1).$

Driver sets of type I have to be zero forcing sets \cite{Monshizadeh 2014}.
It is obvious that $\{1\}$ and $\{n\}$ are the only zero forcing sets for $%
P_{n}$ and that this is true for all $n\geq 2.$ It is easy to see that the
orbit $\{\{1\},\{n\}\}$ is of type I without resorting to the notion of zero
forcing sets. We only need to show this for one representative of the orbit.
For each $X=\left[ x_{ij}\right] \in $ Sym$(P_{n})$ the matrix $(X-\lambda
I)_{\{2,\ldots ,n\}}$ is an echelon matrix with $n-1$ pivots $%
x_{12},x_{23},\ldots ,x_{n-1,n}$ hence rank $(X-\lambda I)_{\{2,\ldots
,n\}}=n-1$ for all $X\in $ Sym$(P_{n})$ and $\lambda \in \mathbb{C}.$ Before
examing the other orbits of minimal driver sets we present some useful
lemmas.

\begin{lemma}
\label{determinant formula}For each $X=[x_{ij}]\in $ Sym$_{0}(P_{n})$ we
have
\begin{equation*}
\det X=\left\{
\begin{array}{cc}
0\text{ } & \text{if }n\text{ is odd} \\
(-1)^{\frac{n}{2}}x_{12}^{2}x_{34}^{2}\cdots x_{n-1,n}^{2} & \text{if }n%
\text{ is even}%
\end{array}%
\right.
\end{equation*}

\begin{proof}
Let $d_{n}=\det X$ with $X=[x_{ij}]\in $ Sym$_{0}(P_{n}).$ Then $d_{1}=0$
and $d_{2}=-x_{12}^{2}$ and expansion along the last column and then along
the last row yields the recurrence relation
\begin{equation*}
d_{n}=-x_{n-1,n}^{2}d_{n-2}
\end{equation*}%
for all $n\geq 3.$
\end{proof}
\end{lemma}

For each $X\in $ Sym$_{0}(P_{n})$ with $n\geq 3$ and $i\in \{2,\ldots ,n-1\}$
the matrix $X_{V\backslash \{i\}}$ has the block structure%
\begin{equation}
\left[
\begin{tabular}{c|c|c}
$Y$ & $%
\begin{array}{c}
0 \\
\vdots \\
0 \\
x_{i-1,i}%
\end{array}%
$ & $0$ \\ \hline
$0$ & $%
\begin{array}{c}
x_{i,i+1} \\
0 \\
\vdots \\
0%
\end{array}%
$ & $Z$%
\end{tabular}%
\right]  \label{block structure Sym0(Pn)}
\end{equation}%
with $Y\in $ Sym$_{0}(P_{i-1})$ and $Z\in $ Sym$_{0}(P_{n-i}).$

\begin{lemma}
\footnote{%
A similar result has been proved in \cite{Parlangeli 2012} with respect to
the system $(L_{n},B_{\{i\}})$, where $L_{n}$ is the Laplacian matrix of $%
P_{n}.$}\label{common eigenvalue 1}Let $X\in $ Sym$_{0}(P_{n})$ with $n\geq
3 $ and $i\in \{2,\ldots ,n-1\}$ and $Y$ and $Z$ as in (\ref{block structure
Sym0(Pn)}). Then rank $(X-\lambda I)_{V\backslash \{i\}}<n-1$ if and only if
$Y$ and $Z$ have a common eigenvalue $\lambda .$

\begin{proof}
The linear system $(X-\lambda I)_{V\backslash \{i\}}^{T}\mathbf{v}=\mathbf{0}
$ breaks down into%
\begin{equation*}
\begin{tabular}{l}
$(1)$ $(Y-\lambda I)\mathbf{v}_{\{1,2,\ldots ,i-1\}}=\mathbf{0}$ \\
$(2)\text{ }x_{i-1,i}v_{i-1}+x_{i,i+1}v_{i}=0$ \\
$(3)$ $(Z-\lambda I)\mathbf{v}_{\{i+1,i+2,\ldots ,n-1\}}=\mathbf{0}$%
\end{tabular}%
\end{equation*}%
Equation $(2)$ implies that either $v_{i-1}=v_{i}=0$ or $v_{i-1}v_{i}\neq 0.$
If $v_{i-1}=v_{i}=0$ then it follows from $(1)$ and $(3)$ that $\mathbf{v}=%
\mathbf{0}$. Suppose rank $(X-\lambda I)_{V\backslash \{i\}}<n-1,$ i.e.,
suppose the system above does have a non-trivial solution $\mathbf{v.}$ Then
$v_{i-1}\neq 0$ and $v_{i}\neq 0$ hence $\mathbf{v}_{\{1,2,\ldots
,i-1\}}\neq $ $\mathbf{0}$ and $\mathbf{v}_{\{i+1,i+2,\ldots ,n-1\}}\neq
\mathbf{0}$, so $(1)$ and $(3)$ show that $\lambda $ is an eigenvalue of $Y$
and $Z$. Conversely, suppose $Y$ and $Z$ have a common eigenvalue $\lambda ,$
i.e., suppose $(1)$ and $(3)$ have non-trivial solutions. These solutions
can be scaled in a such a way that $v_{i-1}$ and $v_{i}$ satisfy equation $%
(2),$ hence a non-trivial solution of the linear system $(X-\lambda
I)_{V\backslash \{i\}}^{T}\mathbf{v}=\mathbf{0}$ exists.
\end{proof}
\end{lemma}

Now let us examine the orbit $\{\{2\},\{n-1\}\}.$ Due to Theorem \ref{driver
set Pn} $\{2\}$ is a driver set if and only if gcd$(2,n+1)=1,$ i.e., if and
only if $n$ is even. It is a zero forcing set for $n=2,$ so we consider $%
n\geq 4.$

\begin{theorem}
For all even $n\geq 4$ the minimal driver sets $\{2\}$ and $\{n-1\}$ for the
graph $P_{n}$ are of type II.

\begin{proof}
We only need to show this for one representative of the orbit. Due to Lemma %
\ref{common eigenvalue 1} $(P_{n},\{2\})$ is not strongly Sym$_{0}(P_{n})$%
-controllable if and only if there exists an $X\in $ Sym$_{0}(P_{n})$ such
that $Y\in $ Sym$_{0}(P_{1})$ and $Z\in $ Sym$_{0}(P_{n-2})$ (as defined in (%
\ref{block structure Sym0(Pn)})) have a common eigenvalue. In this case $%
Y=[0]$ so $(P_{n},\{2\})$ is not strongly Sym$_{0}(P_{n})$-controllable if
and only if $Z$ is singular. It follows from Lemma \ref{determinant formula}
that $\det Z=0$ if and only if $n-2$ is odd.
\end{proof}
\end{theorem}

Finally we show that the remaining orbits are not of type II.

\begin{theorem}
Let $\{i\}$ be \ a minimal driver set for $P_{n}$ with $3\leq i\leq n-2.$ $%
(P_{n},\{i\})$ is not strongly Sym$_{0}(P_{n})$-controllable.

\begin{proof}
Due to Lemma \ref{common eigenvalue 1} $(P_{n},\{i\})$ is not strongly Sym$%
_{0}(P_{n})$-controllable if and only if there exists an $X\in $ Sym$%
_{0}(P_{n})$ such that $Y$ and $Z$ have a common eigenvalue. For each $i\in
\{3,\ldots ,n-2\}$ such a pair $Y,Z$ is easily constructed in the following
way. Choose any $Y\in $ Sym$_{0}(P_{i-1})$ and $Z\in $ Sym$_{0}(P_{n-i})$
and a pair $\lambda _{0},\mu _{0}$ of non-zero eigenvalues of $Y$ and $Z$
respectively. Then $\mu _{0}Y\in $ Sym$_{0}(P_{i-1})$ and $\lambda _{0}Z\in $
Sym$_{0}(P_{n-i})$ share the eigenvalue $\lambda _{0}\mu _{0}.$
\end{proof}
\end{theorem}

\section{Cycle graphs}

Let $\omega =\exp (i\frac{2\pi }{n}).$ The eigenvalues of the adjacency
matrix $A=A(C_{n})$ are given by $\lambda _{k}=\omega ^{k}+\omega
^{n-k}=2\cos \left( \frac{2k\pi }{n}\right) $ with $k=0,2,\ldots ,n-1.$ The
algebraic multiplicities (equal to the geometric ones because $A$ is
symmetric hence diagonalizable) are all equal to 2 with the exceptions of $%
\lambda _{0}=2$ for all $n$ and $\lambda _{\frac{n}{2}}=-2$ for all even $n.$
Hence $M(C_{n})=2$ which implies $D(C_{n})\geq 2.$ On the other hand $%
Z(C_{n})=2$ hence $D(C_{n})=2$ as well. The following Theorem specifies
which pairs of vertices do in fact form a minimal driver set.

\begin{theorem}
$\{i,j\}$ is a driver set for the graph $C_{n}$ if and only if%
\begin{equation*}
\text{gcd}(2d,n)\in \{1,2\},
\end{equation*}%
where $d=d(i,j)$ denotes the distance between the vertices $i$ and $j.$

\begin{proof}
The 1-dimensional eigenspaces of $A$ is/are given by%
\begin{eqnarray*}
&&\text{Span }\left\{ \left[
\begin{array}{cccc}
1 & 1 & \cdots  & 1%
\end{array}%
\right] ^{T}\right\} \text{ for all }n\text{ and } \\
&&\text{Span }\left\{ \left[
\begin{array}{ccccc}
1 & -1 & \cdots  & 1 & -1%
\end{array}%
\right] ^{T}\right\} \text{ for all even }n,
\end{eqnarray*}%
hence all entries of the eigenvectors from these eigenspaces are unequal to $%
0$. Therefore it follows from Lemma \ref{Plucker} that $\{i,j\}$ is not a
driver set if and only the Pl\"{u}cker coordinate $p_{ij}$ of at least one
2-dimensional eigenspace of $A$ is equal to zero. Let $\lambda _{k}$ be an
eigenvalue\ of $A$ of multiplicity $2$ of $A,$ i.e., let%
\begin{equation*}
\lambda _{k}=\omega ^{k}+\omega ^{n-k},
\end{equation*}%
with $\omega =\exp (i\frac{2\pi }{n})$ for $k\in \{1,\ldots ,n\}$ except $k=%
\frac{n}{2}$ if $n$ is even$.$ The corresponding $2$-dimensional eigenspace $%
E_{\lambda _{k}}$ is given by
\begin{equation*}
E_{\lambda _{k}}=\text{ Span }\left\{ \mathbf{v}_{k},\mathbf{\bar{v}}%
_{k}\right\} ,
\end{equation*}%
where $\mathbf{v}_{k}=\left[
\begin{array}{cccc}
1 & \omega ^{k} & \cdots  & \omega ^{(n-1)k}%
\end{array}%
\right] ^{T}.$ The Pl\"{u}cker coordinate $p_{ij}$ of $E_{\lambda }$ is
given by%
\begin{equation*}
p_{ij}=\det \left[
\begin{array}{cc}
(\mathbf{v}_{k})_{i} & (\mathbf{\bar{v}}_{k})_{i} \\
(\mathbf{v}_{k})_{j} & (\mathbf{\bar{v}}_{k})_{j}%
\end{array}%
\right] ,
\end{equation*}%
hence%
\begin{equation*}
p_{ij}=\det \left[
\begin{array}{cc}
\omega ^{(i-1)k} & \omega ^{n-(i-1)k} \\
\omega ^{(j-1)k} & \omega ^{n-(j-1)k}%
\end{array}%
\right] =\omega ^{n+k(i-j)}-\omega ^{n+k((j-i)}.
\end{equation*}%
Now $\omega ^{n+k(i-j)}-\omega ^{n+k((j-i)}=0$ if and only if $k(i-j)\equiv
k(j-i)$ mod $n,$ i.e., if and only if $2k(j-i)\equiv 0$ mod $n$. The trivial
solutions $k=0$ and $k=\frac{n}{2}$ for even $n$ do not correspond to a
2-dimensional eigenspace hence $\{i,j\}$ is not a driver set if and only if
gcd$(2(j-i),n)\notin \{1,2\}.$ This condition can be replaced by gcd$%
(2d,n)\notin \{1,2\}$ because $d(i,j)=\min \{j-i,n+i-j)$ and $2k(j-i)\equiv 0
$ mod $n$ if and only if $2k(n+i-j)\equiv 0$ mod $n.$
\end{proof}
\end{theorem}

The orbits of minimal driver sets under the group $Aut(C_{n})\cong D_{n}$
are the sets $\Omega _{d}$ defined by
\begin{equation*}
\Omega _{d}=\{\{i,j\}\in \binom{V}{2}\text{ }|\text{ }d(i,j)=d\}
\end{equation*}%
for fixed values of $d\in \{1,\ldots ,\left\lfloor \frac{n}{2}\right\rfloor
\}$ satisfying gcd$(2d,n)\in \{1,2\}.$ Since the size of each orbit is equal
to $n$ the number of orbits is equal to $\frac{1}{n}N_{D}(C_{n}).$Note that
the value $d=1$ satisfies gcd$(2d,n)\in \{1,2\}$ for all $n$ and that the
corresponding orbit $\Omega _{1}$ is the (unique) orbit of minimal driver
sets that are zero forcing sets. The following table lists the values of $%
N_{D}(C_{n})$ for $n\leq 12.$%
\begin{equation*}
\begin{tabular}{|c|c|c|c|c|c|c|c|c|c|c|}
\hline
$n$ & $3$ & $4$ & $5$ & $6$ & $7$ & $8$ & $9$ & $10$ & $11$ & $12$ \\ \hline
$N_{D}(C_{n})$ & $3$ & $4$ & $10$ & $12$ & $21$ & $16$ & $27$ & $40$ & $55$
& $24$ \\ \hline
\end{tabular}%
\end{equation*}

It is easy to see that the orbit $\Omega _{1}$ is of type I without
resorting to the notion of zero forcing sets. We only need to show this for
one representative of the orbit. For each $X=\left[ x_{ij}\right] \in $ Sym$%
(C_{n})$ the matrix $(X-\lambda I)_{\{3,\ldots ,n\}}$ is row equivalent to
an echelon form with pivots $x_{1n},x_{23},x_{34},\ldots ,x_{n-2,n-1}$ hence
rank $(X-\lambda I)_{\{3,\ldots ,n\}}=n-2$ for all $X\in $ Sym$(C_{n})$ and $%
\lambda \in \mathbb{C}.$ Before examing the other orbits we discuss the
analogue of Lemma \ref{common eigenvalue 1} for the cycle graphs. For each $%
X\in $ Sym$_{0}(C_{n})$ with $n\geq 6$ and $j\in \{3,\ldots ,\left\lfloor
\frac{n}{2}\right\rfloor \}$ the matrix $X_{V\backslash \{1,j\}}$ has the
block structure%
\begin{equation}
\left[
\begin{tabular}{c|c|c|c}
$%
\begin{array}{c}
x_{12} \\
0 \\
\vdots \\
\end{array}%
$ & $Y$ & $%
\begin{array}{c}
0 \\
\vdots \\
0 \\
x_{i-1,i}%
\end{array}%
$ & $0$ \\ \cline{2-4}
$%
\begin{array}{c}
\\
\\
0 \\
x_{1n}%
\end{array}%
$ & $0$ & $%
\begin{array}{c}
x_{i,i+1} \\
0 \\
\vdots \\
0%
\end{array}%
$ & $Z$%
\end{tabular}%
\right]  \label{block structure Sym0(Cn)}
\end{equation}%
with $Y\in $ Sym$_{0}(P_{j-2})$ and $Z\in $ Sym$_{0}(P_{n-j}).$ The
following Lemma can be proved in the same way as Lemma \ref{common
eigenvalue 1}.

\begin{lemma}
\footnote{%
A similar result has been proved in \cite{Parlangeli 2012} with respect to
the system $(L_{n},B_{\{i,j\}})$, where $L_{n}$ is the Laplacian matrix of $%
C_{n}.$}\label{common eigenvalue 2}Let $X\in $ Sym$_{0}(C_{n})$ with $n\geq
6 $, $j\in \{3,\ldots ,\left\lfloor \frac{n}{2}\right\rfloor \}$ and $Y$ and
$Z $ as in (\ref{block structure Sym0(Cn)}). Then rank $(X-\lambda
I)_{V\backslash \{1,j\}}<n-2$ if and only if $Y$ and $Z$ have a common
eigenvalue $\lambda .$

\begin{proof}
The linear system $(X-\lambda I)_{V\backslash \{1,j\}}^{T}\mathbf{v}=\mathbf{%
0}$ breaks down into%
\begin{equation*}
\begin{tabular}{l}
$(1)$ $x_{12}v_{1}+x_{1n}v_{n-2}=0$ \\
$(2)$ $(Y-\lambda I)\mathbf{v}_{\{2,\ldots ,j-1\}}=\mathbf{0}$ \\
$(3)\text{ }x_{j-1,j}v_{j-1}+x_{j,j+1}v_{j}=0$ \\
$(4)$ $(Z-\lambda I)\mathbf{v}_{\{j+1,\ldots ,n-2\}}=\mathbf{0}$%
\end{tabular}%
\end{equation*}%
The existence of a non-trivial solution $\mathbf{v}$ forces $%
v_{1},v_{j-1},v_{j}$ and $v_{n-2}$ to be non-zero and $\lambda $ to be an
eigenvalue of $Y$ and $Z.$ Conversely, the existence of non-trivial
solutions of $(2)$ and $(3)$ gives rise to a non-trivial solution $\mathbf{v}
$.
\end{proof}
\end{lemma}

Now let us examine the orbit $\Omega _{2}.$ Due to the theorem above, $%
\{i,j\}$ with $d(i,j)=2$ is a driver set if and only if gcd$(4,n)\in
\{1,2\}. $

\begin{theorem}
$\Omega _{2}$ is a type II orbit of minimal driver sets for $C_{n}$ if and
only if $n$ is odd $(>3).$
\end{theorem}

\begin{proof}
We only need to show this for one representative of the orbit. We consider $%
S=\{i,j\}=\{1,3\}.$ Due to Lemma \ref{common eigenvalue 2} $(C_{n},\{1,3\})$
is not strongly Sym$_{0}(C_{n})$-controllable if and only if there exists an
$X\in $ Sym$_{0}(C_{n})$ such that $Y\in $ Sym$_{0}(P_{1})$ and $Z\in $ Sym$%
_{0}(P_{n-3})$ (as defined in (\ref{block structure Sym0(Cn)})) have a
common eigenvalue. In this case $Y=[0]$ so $(C_{n},\{1,3\})$ is not strongly
Sym$_{0}(C_{n})$-controllable if and only if $Z$ is singular. It follows
from Lemma \ref{determinant formula} that $\det Z\neq 0$ for all odd $n$.
Obviously the case $n=3$ is not included because $\{1,3\}\in \Omega _{1}$
for the graph $C_{3}.$
\end{proof}

Finally we show that the remaining orbits are not of type II.

\begin{theorem}
Let $\Omega _{d}$ be an orbit of minimal driver sets for $C_{n}$ with $d\geq
3.$ $\Omega _{d}$ is not strongly Sym$_{0}(C_{n})$-controllable.
\end{theorem}

\begin{proof}
We only need to show this for one representative of the orbit. We consider $%
\{1,j\}$ with $j\in \{4,\ldots ,\left\lfloor \frac{n}{2}\right\rfloor \}.$
Due to Lemma \ref{common eigenvalue 2} $(C_{n},\{1,j\})$ is not strongly Sym$%
_{0}(C_{n})$-controllable if and only if there exists an $X\in $ Sym$%
_{0}(C_{n})$ such that $Y$ and $Z$ have a common eigenvalue. Choose any $%
Y\in $ Sym$_{0}(P_{j-2})$ and $Z\in $ Sym$_{0}(P_{n-j})$ and a pair $\lambda
_{0},\mu _{0}$ of non-zero eigenvalues of $Y$ and $Z$ respectively. Then $%
\mu _{0}Y\in $ Sym$_{0}(P_{j-2})$ and $\lambda _{0}Z\in $ Sym$_{0}(P_{n-j})$
share the eigenvalue $\lambda _{0}\mu _{0}.$
\end{proof}

\end{document}